\documentclass[12pt]{amsart}
\usepackage{amsmath,amssymb}
\usepackage{amsfonts}
\usepackage{amsthm}
\usepackage{latexsym}
\usepackage{graphicx}

\def\lf{\left}
\def\ri{\right}

\def\p{\partial}

\def\R{\mathbb{R}}

\def\pd#1#2{\frac{\partial #1}{\partial #2}}
\def\vv<#1>{\langle#1\rangle}
\def\ol{\overline}

\def\XXint#1#2{\setbox0=\hbox{$#1{#2}{\int}$}{#2}\kern-.5\wd0 }

\def\XXint#1#2#3{{\setbox0=\hbox{$#1{#2#3}{\int}$}
     \vcenter{\hbox{$#2#3$}}\kern-.5\wd0}}

\def\vv<#1>{{\left\langle#1\right\rangle}}
\def\br#1{{\bar{#1}}}
\newtheorem{thm}{Theorem}[section]

\newtheorem{lem}{Lemma}[section]

\newtheorem{cor}{Corollary}[section]
\theoremstyle{definition}

\theoremstyle{remark}

\numberwithin{equation}{section}

\begin{document}
\title{Line segment energy and applications}

\author{Yulian Chen}
\address{Department of Mathematics, Shantou University, Shantou, Guangdong, 515063, China}
\email{13ylchen3@stu.edu.cn}
\author{Chengjie Yu$^1$}
\address{Department of Mathematics, Shantou University, Shantou, Guangdong, 515063, China}
\email{cjyu@stu.edu.cn}
\thanks{$^1$Research partially supported by a supporting project from the Department of Education of Guangdong Province with contract no. Yq2013073, and NSFC 11571215.}
\renewcommand{\subjclassname}{%
  \textup{2010} Mathematics Subject Classification}
\date{}
\begin{abstract}
In this paper, we compute the derivatives of the line segment energy for a symmetric tensor field and apply them to obtain slightly more general log-concavity estimates for positive solutions of heat equations and first eigenfunctions on bounded strictly convex domains.
\end{abstract}
\maketitle\markboth{Line segment energy}{ Chen \& Yu}
\section{Introduction}
In the celebrated proof of the fundamental conjecture by Andrews and Clutterbuck \cite{AC}, a sharp log-concavity estimate of the first eigenfunction plays an important role. In the proof, Andrews and Clutterbuck did not try to estimate the Hessian of the logarithmic of the first eigenfunction directly. Instead, they estimate the integration of the Hessian along line segments which they called modulus of expansion. This was observed by Ni \cite{Ni} and was called energy of line segments in \cite{Ni}. This will make things more complicated at first glance because we have doubled the number of spatial variables. However, by applying a clever trick to the new quantity, things become simpler (See \cite{An}). The method has been proved to be important in obtaining sharp estimates by its successes in gradient estimate, eigenvalue estimate etc.(See \cite{An}).

Let $\tau$ be a symmetric (0,2)-tensor on $\R^n$. Then the line segment energy between $x$ and $y$ is defined to be
\begin{equation}
E_\tau(x,y)=\int_0^{\|y-x\|}\tau(\theta(s,x,y))(N(x,y),N(x,y))ds
\end{equation}
where $N(x,y)=\frac{y-x}{\|y-x\|}$  and $\theta(s,x,y)=x+sN(x,y)$. When $f(x)$ is a function, we simply denote $E_{\nabla^2f}(x,y)$ as $E_f(x,y)$. When $f(x,t)$ is a function also depending on time, we simply denote $E_{\nabla^2f(\cdot,t)}(x,y)$ as $E_f(x,y,t)$. The line segment energy also appears in integral geometry where it is called ray transformation of the tensor field $\tau$(See \cite{Sh}).

In this paper, we compute the derivatives of the line segment energy for a symmetric tensor field and apply it to obtain slightly more general log-concavity estimates for positive solutions to heat equations and first eigenfunctions on bounded strictly convex domains. Our main result is as follows.
\begin{thm}\label{thm-main}
Let $\Omega$ be a bounded strictly convex domain in $\R^n$ with smooth boundary and diameter $D$, $q(x,t)\in C^\infty(\overline \Omega\times [0,T])$, $u(x,t)$ be a positive solution of the heat equation
$$\frac{\partial u}{\partial t}-\Delta u+qu=0$$
on $\Omega$ with Dirichlet boundary data. Let $\psi\in C^\infty([0,D/2]\times [0,T])$ with $\psi(0,t)=0$, $\psi_{ss}(0,t)=0$ and $\psi_s>0$. Here  $\psi_s$ means taking derivative with respect to $s$. Suppose that
\begin{equation}
E_q(x,y,t)\geq \varphi(r(x,y)/2,t)
\end{equation}
where $\varphi$ is a nonnegative function on $[0,D/2]\times[0,T]$,
\begin{equation}
E_{f}(x,y,0)\geq m_0\psi(r(x,y)/2,0)
\end{equation}
for some $m_0\geq 0$ where $f=-\log u$. Then
\begin{equation}
E_f(x,y,t)\geq m\psi(r(x,y)/2,t)
\end{equation}
for any $x,y\in \Omega$ and $t\in [0,T]$, where
\begin{equation}\label{eqn-m}
m=\min_{(s,t)\in [0,D/2]\times [0,T]}\left\{m_0,\frac{\psi_{ss}-\psi_t+\sqrt{(\psi_{ss}-\psi_t)^2+4\varphi\psi\psi_s}}{2\psi_s\psi}(s,t)\right\}.
\end{equation}
\end{thm}

By similar arguments, an elliptic version of the last theorem is also obtained (See Theorem \ref{thm-main-elliptic}). As corollaries of the log-concavity estimates, we obtain log-concavity comparisons for positive solutions of heat equations and first eigenfuncions. These log-concavity comparisons are presented in different forms in \cite{AC,Ni}. Our proofs of Theorem \ref{thm-main} apply the very interesting technique in the elliptic proof of fundamental gap theorem in \cite{Ni} where maximum principle is applied to a quotient quantity instead of a difference quantity. Because the quotient quantity has singularity on the diagonal, a process similar to blowing up is used.

The remaining parts of the paper are organized as follows. In Section 2, we compute the derivatives of line segment energy. In Section 3, we apply the computation in Section 2 to log-concavity estimates.
\section{Line segment energy in Euclidean spaces}
Let $\Omega$ be a convex domain in $\R^n$. Let $r(x,y)=\|x-y\|$, $N(x,y)=\frac{y-x}{\|y-x\|}$ and $\theta(s,x,y)=x+sN(x,y)$. It is clear that
$\theta(s,x,y)$ is the line segment joining $x$ to $y$.
Let $\tau$ be a symmetric $(0,2)$-tensor field on $\Omega$. Define the energy of $\tau$ between $x$ and $y$ as
$$E_{\tau}(x,y)=\int_0^{r(x,y)}\tau(\theta(s,x,y))(N(x,y),N(x,y))ds.$$

By direct computation, we can obtain the following first and second derivatives of $r(x,y),N(x,y)$ and $\theta(s,x,y)$.

\begin{lem}\label{lm-com}
Let $x_0$ and $y_0$ be two different points in $\R^n$, and $e_1,e_2,\cdots,e_n$ be an othonormal frame with $e_n=N(x_0,y_0)$. Let $E_i=(e_i,e_i)$ and $\tilde E_i=(e_i,-e_i)$ for $i=1,2,\cdots,n$. Then
\begin{enumerate}
\item $\nabla_{E_i}N(x_0,y_0)=0,\ \nabla_{E_i}r(x_0,y_0)=0,\ \nabla_{E_i}\theta(s,x_0,y_0)=e_i$ for $i=1,2,\cdots n$;
\item$\nabla_{\tilde E_i}N(x_0,y_0)=-\frac{2}{r(x_0,y_0)}e_i,\ \nabla_{\tilde E_i}r(x_0,y_0)=0,\ \nabla_{\tilde E_i}\theta(s,x_0,y_0)=\lf(1-\frac{2s}{r(x_0,y_0)}\ri)e_i$ for $i<n$;
\item $\nabla_{\tilde E_n}N(x_0,y_0)=0,\ \nabla_{\tilde E_n}r(x_0,y_0)=-2,\ \nabla_{\tilde E_n}\theta(s,x_0,y_0)=e_n$;
\item $\nabla_{E_i}\nabla_{E_i}N(x_0,y_0)=\nabla_{E_i}\nabla_{E_i}r(x_0,y_0)=\nabla_{E_i}\nabla_{E_i}\theta(s,x_0,y_0)=0$
for $i=1,2,\cdots,n$.
\item $\nabla_{\tilde E_n}\nabla_{\tilde E_n}N(x_0,y_0)=\nabla_{\tilde E_n}\nabla_{\tilde E_n}r(x_0,y_0)=\nabla_{\tilde E_n}\nabla_{\tilde E_n}\theta(s,x_0,y_0)=0$;
\end{enumerate}
\end{lem}
\begin{proof}
We only compute $\nabla_{E_i}\nabla_{E_i}N(x_0,y_0)$ in (4), the others are similar.
\begin{equation}
\nabla_{E_i}\nabla_{E_i}N(x_0,y_0)=\frac{d^2}{dt^2}N(x_0+te_i,y_0+te_i)\bigg|_{t=0}=\frac{d^2}{dt^2}N(x_0,y_0)\bigg|_{t=0}=0.
\end{equation}
\end{proof}
We now come to compute the derivatives of $E_\tau$.
\begin{thm}\label{thm-1-order-1}Let notations be the same as before. Then,
$$\nabla_{E_i}E_{\tau}(x_0,y_0)=E_{\nabla_{e_i}\tau}(x_0,y_0)$$
for $i=1,2,\cdots,n$.
\end{thm}
\begin{proof}By (1) in Lemma \ref{lm-com},
\begin{equation}
\begin{split}
&\nabla_{E_i}E_{\tau}(x_0,y_0)\\
=&\nabla_{E_i}\left(\int_0^{r(x,y)}\tau_{\alpha\beta}(\theta(s,x,y))N_\alpha N_\beta(x,y) ds\right)\bigg|_{x=x_0,y=y_0}\\
=&\int_0^{r(x_0,y_0)}\tau_{\alpha\beta;\gamma}(\theta(s,x_0,y_0))\nabla_{E_i}\theta_\gamma(s,x_0,y_0)N_\alpha N_\beta(x_0,y_0) ds\\
=&\int_0^{r(x_0,y_0)}\tau_{\alpha\beta;i}(\theta(s,x_0,y_0))N_\alpha N_\beta(x_0,y_0)ds\\
=&E_{\nabla_{e_i}\tau}(x_0,y_0).
\end{split}
\end{equation}
\end{proof}
\begin{thm}\label{thm-1-order-2}Let notations be the same as before. Then
\begin{equation}
\begin{split}
&\nabla_{\tilde E_i}E_\tau(x_0,y_0)\\
=&\left\{\begin{array}{ll}E_{\nabla_{e_i}\tau}(x_0,y_0)-\frac{2}{r(x_0,y_0)}\int_0^{r(x_0,y_0)}(s\tau_{nn;i}+2\tau_{in})(\theta(s,x_0,y_0))ds&i<n\\
E_{\nabla_{e_n}\tau}(x_0,y_0)-2\tau_{nn}(y_0)&i=n.
\end{array}\right.
\end{split}
\end{equation}
\end{thm}
\begin{proof}
By (2) of Lemma \ref{lm-com}, we have
\begin{equation}
\begin{split}
&\nabla_{\tilde E_i}E_\tau(x_0,y_0)\\
=&\nabla_{\tilde E_i}\lf(\int_0^{r(x,y)}\tau_{\alpha\beta}(\theta(s,x,y))N_\alpha N_\beta(x,y)ds \ri)\bigg|_{x=x_0,y=y_0}\\
=&\int_0^{r(x_0,y_0)}\tau_{\alpha\beta;\gamma}(\theta(s,x_0,y_0))\nabla_{\tilde E_i}\theta_\gamma(s,x_0,y_0)N_\alpha N_\beta(x_0,y_0)ds\\
&+2\int_0^{r(x_0,y_0)}\tau_{\alpha\beta}(\theta(s,x_0,y_0))\nabla_{\tilde E_i}N_\alpha N_\beta(x_0,y_0)ds\\
=&E_{\nabla_{e_i}\tau}(x_0,y_0)-\frac{2}{r(x_0,y_0)}\int_0^{r(x_0,y_0)}(s\tau_{nn;i}+2\tau_{in})(\theta(s,x_0,y_0))ds
\end{split}
\end{equation}
when $i<n$.

By (3) of Lemma \ref{lm-com}, we have
\begin{equation}
\begin{split}
&\nabla_{\tilde E_n}E_\tau(x_0,y_0)\\
=&\nabla_{\tilde E_n}\lf(\int_0^{r(x,y)}\tau_{\alpha\beta}(\theta(s,x,y))N_\alpha N_\beta(x,y)ds \ri)\bigg|_{x=x_0,y=y_0}\\
=&-2\tau_{nn}(y_0)+\int_0^{r(x_0,y_0)}\tau_{\alpha\beta;\gamma}(\theta(s,x_0,y_0))\nabla_{\tilde E_n}\theta_\gamma(s,x_0,y_0)N_\alpha N_\beta(x_0,y_0)ds\\
=&E_{\nabla_{e_n}\tau}(x_0,y_0)-2\tau_{nn}(y_0).
\end{split}
\end{equation}
\end{proof}
\begin{thm}\label{thm-2-oder-1}Let notations be the same as before. Then
$$\nabla_{E_i}\nabla_{E_i}E_{\tau}(x_0,y_0)=E_{\nabla_{e_i}\nabla_{e_i}\tau}(x_0,y_0)$$
for $i=1,2,\cdots,n.$
\end{thm}
\begin{proof}By (1) and (4) of Lemma \ref{lm-com}, we have
\begin{equation}
\begin{split}
&\nabla_{E_i}\nabla_{E_i}E_{\tau}(x_0,y_0)\\
=&\nabla_{E_i}\nabla_{E_i}\left(\int_0^{r(x,y)}\tau_{\alpha\beta}(\theta)N_\alpha N_\beta ds\right)\bigg|_{x=x_0,y=y_0}\\
=&\nabla_{E_i}\left(\tau_{\alpha\beta}(y)N_\alpha N_\beta\nabla_{E_i}r+\int_0^{r(x,y)}\nabla_{E_i}(\tau_{\alpha\beta}(\theta)N_\alpha N_\beta) ds\right)\bigg|_{x=x_0,y=y_0}\\
=&\int_0^{r(x_0,y_0)}\nabla_{E_i}\nabla_{E_i}(\tau_{\alpha\beta}(\theta)N_\alpha N_\beta) ds\\
=&\int_0^{r(x_0,y_0)}\tau_{\alpha\beta;kl}(\theta)\nabla_{E_i}\theta_l\nabla_{E_i}\theta_kN_\alpha N_\beta ds\\
=&\int_0^{r(x_0,y_0)}\nabla_{e_i}\nabla_{e_i}\tau_{\alpha\beta}(\theta)N_\alpha N_\beta ds\\
=&E_{\nabla_{e_i}\nabla_{e_i}\tau}(x_0,y_0).
\end{split}
\end{equation}
\end{proof}

\begin{thm}\label{thm-2-order-2}Let notations be the same as before. Then,
$$\nabla_{\tilde E_n}\nabla_{\tilde E_n}E_{\tau}(x_0,y_0)=E_{\nabla_{e_n}\nabla_{e_n}\tau}(x_0,y_0).$$
\end{thm}
\begin{proof} By (3) and (5) of Lemma \ref{lm-com},
\begin{equation}
\begin{split}
&\nabla_{\tilde E_n}\nabla_{\tilde E_n}E_{\tau}(x_0,y_0)\\
=&\nabla_{\tilde E_n}\nabla_{\tilde E_n}\left(\int_0^{r(x,y)}\tau_{\alpha\beta}(\theta)N_\alpha N_\beta ds\right)\bigg|_{x=x_0,y=y_0}\\
=&\nabla_{\tilde E_n}\left(\tau_{\alpha\beta}(y)N_\alpha N_\beta\nabla_{\tilde E_n}r+\int_0^{r(x,y)}\nabla_{\tilde E_n}(\tau_{\alpha\beta}(\theta)N_\alpha N_\beta) ds\right)\bigg|_{x=x_0,y=y_0}\\
=&2\tau_{nn;n}(y_0)+\nabla_{\tilde E_n}\left(\int_0^{r(x,y)}(\tau_{\alpha\beta;\gamma}(\theta)\nabla_{\tilde E_n}\theta_\gamma N_\alpha N_\beta) ds\right)\bigg|_{x=x_0,y=y_0}\\
=&\int_0^{r(x_0,y_0)}\tau_{\alpha\beta;\gamma\delta}(\theta)\nabla_{\tilde E_n}\theta_\delta\nabla_{\tilde E_n}\theta_\gamma N_\alpha N_\beta ds\\
=&\int_0^{r(x_0,y_0)}\nabla_{e_n}\nabla_{e_n}\tau_{\alpha\beta}(\theta)N_\alpha N_\beta ds\\
=&E_{\nabla_{e_n}\nabla_{e_n}\tau}(x_0,y_0).
\end{split}
\end{equation}
\end{proof}

\section{Applications to log-concavity comparison}
In this section, we use the line segment energy and an interesting technique in \cite{Ni} to derive slightly more general log-concavity estimates. We first need some boundary behaviors for the line segment energy of the logarithmic of a defining function for a bounded strictly convex domain. Before doing this,
we need the following boundary behavior for the Hessian of the logarithmic of a defining function of a bounded strictly convex domain.
\begin{lem}\label{lm-bd}
Let $\Omega$ be a bounded strictly convex domain in $\R^n$ with smooth boundary and $\phi\in C^2(\overline \Omega)$ such that $\phi>0$ in $\Omega$, $\phi\big|_{\p\Omega}=0$ and $\frac{\p\phi}{\p \nu}\big|_{\p\Omega}<0$ where $\nu$ is the unit outward normal of $\Omega$. Let $f=-\log\phi$. Then, there are two positive constants $\delta_0$ and $c_0$, such that
\begin{equation}
\nabla^2f(x)(X,X)\geq \frac{c_0\|X\|^2}{\phi(x)}
\end{equation}
 for any $x\in \Omega$ with $\phi(x)<\delta_0$ and $X\in \R^n$.
\end{lem}
\begin{proof}
For $\delta\geq 0$, let $\Omega_\delta=\{x\in\Omega\ |\ \phi(x)>\delta\}.$
Since $\frac{\p\phi}{\p \nu}|_{\p\Omega}<0$, we have $\nabla\phi|_{\p\Omega}\neq 0$. By coninuity, there is $\delta_1>0$ such that $\nabla\phi\neq 0$ on $\Omega\setminus\Omega_{\delta_1}$. It is clear that, when $\delta<\delta_1$, the unit outward normal of $\Omega_\delta$ is
$$\nu=-\frac{\nabla\phi}{\|\nabla\phi\|}\bigg|_{\p\Omega_\delta}.$$
The second fundamental form of $\p\Omega_\delta$ with respect to $\nu$ is
$$II(X,Y)=-\frac{\nabla^2\phi(X,Y)}{\|\nabla\phi\|}$$
for any $X,Y$ with $\vv<X,\nabla\phi>=\vv<Y,\nabla\phi>=0$. Since $\Omega$ is strictly convex, $II(X,X)>0$ on $\p\Omega$ for any $X$ with $\vv<X,\nabla\phi>=0$. By continuity, there are positive numbers $\delta_2<\delta_1$ and $c_1$  such that
\begin{equation}
-\nabla^2\phi(x)(X,X)\geq c_1\|X\|^2,\ \mbox{and}\ \|\nabla\phi\|^2(x)\geq c_1
\end{equation}
for any $X$ with $\vv<X,\nabla\phi(x)>=0$ and $x$ with $\phi(x)<\delta_2$.

For any $x\in \Omega\setminus \Omega_{\delta_2}$ and vector $X$, let
$$X=X_1+X_2$$
is the orthogonal decomposition of $X$ with $\vv<X_1,\nabla\phi(x)>=0$ and $X_2$ parallel to $\nabla\phi(x)$. Then
\begin{equation}\label{eqn-bnd-Hess-f}
\begin{split}
&\nabla^2f(X,X)\\
=&-\frac{\nabla^2\phi(x)(X,X)}{\phi(x)}+\frac{\vv<\nabla\phi(x),X>^2}{\phi(x)^{2}}\\
=&-\frac{\nabla^2\phi(x)(X_1,X_1)}{\phi(x)}-\frac{2\nabla^2\phi(x)(X_1,X_2)}{\phi(x)}-\frac{\nabla^2\phi(x)(X_2,X_2)}{\phi(x)}+\frac{\vv<\nabla\phi(x),X_2>^2}{\phi(x)^{2}}\\
\geq&\frac{c_1\|X_1\|^2}{\phi(x)}-\frac{2A\|X_1\|\|X_2\|}{\phi(x)}-\frac{A\|X_2\|^2}{\phi(x)}+\frac{c_1\|X_2\|^2}{\phi(x)^2}\\
\geq&\frac{c_1\|X_1\|^2}{2\phi(x)}+\left(\frac{c_1}{\phi(x)^2}-\frac{A+2A^2/c_1}{\phi(x)}\right)\|X_2\|^2\\
\geq&\frac{c_1}{2\phi(x)}\|X_1\|^2+\frac{c_1}{2\phi(x)^2}\|X_2\|^2\\
\geq&\frac{c_1\|X\|^2}{2\phi(x)}
\end{split}
\end{equation}
when $\phi(x)<\delta_3$ be small enough, where $A=\|\phi\|_{C^2(\overline\Omega)}$.
\end{proof}
We are now ready to derive the boundary behavior of the line segment energy for a defining function of a bounded strictly convex domain.
\begin{thm}\label{thm-bd}
Let $\Omega$ be a bounded strictly convex domain in $\R^n$ with smooth boundary and $\phi\in C^2(\overline \Omega)$ such that $\phi>0$ in $\Omega$, $\phi\big|_{\p\Omega}=0$ and $\frac{\p\phi}{\p \nu}\big|_{\p\Omega}<0$ where $\nu$ is the unit outward normal of $\Omega$. Let $f=-\log\phi$, and $\{x_k\}$ and $\{y_k\}$ be two sequences of points in $\Omega$ with $x_k\to x_0$ and $y_k\to y_0$ as $k\to\infty$. Then
\begin{enumerate}
\item when $x_0\in\Omega$ and $y_0\in\partial\Omega$, $\lim_{k\to\infty} E_f(x_k,y_k)=+\infty$;
\item when $x_0\in \partial\Omega$ and $y_0\in\Omega$, $\lim_{k\to\infty} E_f(x_k,y_k)=+\infty$;
\item when $x_0,y_0\in\partial\Omega$ and $x_0\neq y_0$, $\lim_{k\to\infty} E_f(x_k,y_k)=+\infty$;
\item when $x_0=y_0\in\partial\Omega$, $\liminf_{k\to\infty} E_f(x_k,y_k)\geq0$;
\item when $x_0=y_0\in\partial\Omega$ and $x_k\neq y_k$ for $k$  large enough, $$\lim_{k\to\infty} \frac{E_f(x_k,y_k)}{r(x_k,y_k)}=+\infty.$$
\end{enumerate}
\end{thm}
\begin{proof}
\begin{enumerate}
\item Without loss of generality, we can suppose that $x_0\in \Omega_{\delta_0}$ where $\delta_0$ is the same as in Lemma \ref{lm-bd}. Otherwise, we can shrink $\delta_0$ so that $\Omega_{\delta_0}$ contains $x_0$. Hence for $k$ large enough, we have $x_k\in\Omega_{\delta_0}$ and $y_k\not\in\ol\Omega_{\delta_0}$. Let
$$s_k=\sup\{s\in(0,r(x_k,y_k))\ |\ \theta(s,x_k,y_k)\in\Omega_{\delta_0}\}.$$
Then, by convexity of $\Omega_{\delta_0}$, we know that
$$\theta(s,x_k,y_k)\in \ol\Omega_{\delta_0}\ \mbox{for all}\ s\in [0,s_k]$$
and
$$\theta(s,x_k,y_k)\in \Omega\setminus\Omega_{\delta_0}\ \mbox{for all}\ s\in [s_k,r(x_k,y_k)].$$
Hence, by Lemma \ref{lm-bd},
\begin{equation}
\begin{split}
&E_{f}(x_k,y_k)\\
=&\int_0^{r(x_k,y_k)}\nabla^2f(\theta(s,x_k,y_k))(N(x_k,y_k),N(x_k,y_k))ds\\
=&\int_0^{s_k}\nabla^2f(\theta(s,x_k,y_k))(N(x_k,y_k),N(x_k,y_k))ds\\
&+\int_{s_k}^{r(x_k,y_k)}\nabla^2f(\theta(s,x_k,y_k))(N(x_k,y_k),N(x_k,y_k))ds\\
\geq& -\|f\|_{C^2(\ol\Omega_{\delta_0})}s_k+\int_{s_k}^{r(x_k,y_k)}\frac{c_0}{\phi(\theta(s,x_k,y_k))}ds\\
=&-\|f\|_{C^2(\ol\Omega_{\delta_0})}D+c_0\int_{\phi(y_k)}^{\delta_0}\frac{1}{\sigma|\vv<\nabla\phi,N(x_k,y_k)>|}d\sigma\\
\geq&-\|f\|_{C^2(\ol\Omega_{\delta_0})}D+\frac{c_0}{\|\phi\|_{C^1(\bar\Omega)}}\int_{\phi(y_k)}^{\delta_0}\frac{1}{\sigma}d\sigma\\
=&-\|f\|_{C^2(\ol\Omega_{\delta_0})}D+\frac{c_0}{\|\phi\|_{C^1(\ol\Omega)}}(\ln\delta_0-\ln\phi(y_k)).\\
\end{split}
\end{equation}
Here $D$ is the diameter of $\Omega$. Since $\phi(y_k)\to 0$ as $k\to\infty$, we get the conclusion.

\item By the symmetry of $E_f(x,y)$, it is clear.
 \item Without loss of generality, we can assume that the intersection of the line segment $\overline{x_0y_0}$ and $\Omega_{\delta_0}$ is nonempty, otherwise we can shrink $\delta_0$ to make the intersection nonempty. For $k$ large enough such that $x_k,y_k\not\in\overline\Omega_{\delta_0}$, let
$$\sigma_k=\inf\{s\in(0,r(x_k,y_k))\ |\ \theta(s,x_k,y_k)\in\Omega_{\delta_0}\}$$
and $s_k$ the same as before. By the convexity of $\Omega_{\delta_0}$, it is clear that
$$\theta(s,x_k,y_k)\in \ol\Omega_{\delta_0}\ \mbox{for all}\ s\in [\sigma_k,s_k]$$
and
$$\theta(s,x_k,y_k)\in \Omega\setminus \Omega_{\delta_0}\ \mbox{for all}\ s\in [0,\sigma_k]\cup[s_k,r(x_k,y_k)].$$
Then, a similar argument as in (1) gives us
\begin{equation}
E_f(x_k,y_k)\geq -\|f\|_{C^2(\overline{\Omega_{\delta_0}})}D+\frac{c_0}{\|\phi\|_{C^1(\overline\Omega)}}(2\ln\delta_0-\ln\phi(x_k)-\ln\phi(y_k)).
\end{equation}
So the conclusion follows.
 \item When $k$ is large enough, we know that the line segment $\overline{x_ky_k}$ will be contained in $\Omega\setminus\Omega_{\delta_0}$. By Lemma \ref{lm-bd}, we know that the Hessian of $f$ is positive on the line segment $\overline{x_ky_k}$ when $k$ is large enough. Hence, the conclusion follows.
 \item When $k$ is large enough, the line segment $\overline{x_ky_k}$ will be contained in $\Omega\setminus\Omega_{\delta_0}$. Hence, by Lemma \ref{lm-bd}
     \begin{equation}
      \frac{E_{f}(x_k,y_k)}{r(x_k,y_k)}\geq \frac{c_0}{\max_{x\in \overline{x_ky_k}}\phi(x)}
     \end{equation}
     tends to $+\infty$ as $k\to\infty$.
 \end{enumerate}
\end{proof}

We now come to prove Theorem \ref{thm-main}.

\begin{proof}[Proof of Theorem \ref{thm-main}]
Let
$$S=\{(x,y,X,t)\in \Omega\times\Omega\times\mathbb{S}^{n-1}\times[0,T]\ |\ X=N(x,y) \mbox{ if } x\neq y\}.$$
 Define the function $Q(x,y,X,t)$ on $S$ as
 \begin{equation}
Q(x,y,X,t)=\left\{\begin{array}{ll}\frac{E_f(x,y,t)}{\psi(r(x,y)/2,t)}&x\neq y\\
\frac{2\nabla^2f(x,t)(X,X)}{\psi_s(0,t)}&x=y.
\end{array}\right.
\end{equation}
Let $m=\inf_{S} Q$. It is clear that $m<+\infty$. Let $\{(x_k,y_k,X_k,t_k)\}$ be a sequence of points in $S$ such that
$$m=\lim_{k\to\infty}Q(x_k,y_k,X_k,t_k).$$
By passing to a subsequence if necessary, we can assume that $$(x_k,y_k,X_k,t_k)\to(x_0,y_0,X_0,t_0)$$
as $k\to\infty$. By Lemma \ref{lm-bd} and Theorem \ref{thm-bd}, we know that
$$Q(x_k,y_k,X_k,t_k)\to+\infty$$
as $k\to\infty$, when $x_0\in\p\Omega$ or $y_0\in \p\Omega$. This is impossible. Hence $(x_0,y_0)\in \Omega\times\Omega$. By continuity of $Q$, we know that $Q$ achieves its minimum at $(x_0,y_0,X_0,t_0)$. If $t_0=0$, the conclusion follows directly. So, we can suppose that $t_0>0$. Moreover, note that, by Brascamp-Lieb \cite{BL} (See also \cite{CS,Ko,SWYY}), $m>0$. We will prove the conclusion in the following two cases.

(i) When $x_0\neq y_0$, let $e_1,e_2,\cdots,e_n=N(x_0,y_0)$ be an orthonormal basis of $\R^n$, $r_0=r(x_0,y_0)$ and $\theta(s)=\theta(s,x_0,y_0)$. Let $E_i=(e_i,e_i)$ and $\tilde E_i=(e_i,-e_i)$ for $i=1,2,\cdots,n$. Then, by that $$\nabla_{E_i}Q(x_0,y_0,X_0,t_0)=\nabla_{\tilde E_i}Q(x_0,y_0,X_0,t_0)=0,$$ (1),(2),(3) Lemma \ref{lm-com}, Theorem \ref{thm-1-order-1} and Theorem \ref{thm-1-order-2}, we have the follows:
\begin{equation}\label{eqn-first-1}
E_{f_i}(x_0,y_0,t_0)=0
\end{equation}
where $f_i$ means $\nabla_{e_i}f$, for $i=1,2,\cdots,n$,
\begin{equation}\label{eqn-first-2}
\int_0^{r_0}(sf_{nni}+2f_{in})(\theta(s),t_0)ds=0
\end{equation}
for $i=1,2,\cdots,n-1$, and
\begin{equation}\label{eqn-first-3}
f_{nn}(y_0,t_0)=m\psi_s(r_0/2,t_0)/2.
\end{equation}

Furthermore, by \eqref{eqn-first-2},
\begin{equation}\label{eqn-first-4}
\begin{split}
0=&\int_0^{r_0}sf_{inn}(\theta(s),t_0)ds+2\int_0^{r_0}f_{in}(\theta(s),t_0)ds\\
=&\int_0^{r_0}f_{inn}(\theta(s),t_0)\int_0^s d\sigma ds+2\int_0^{r_0}f_{in}(\theta(s),t_0)ds\\
=&\int_0^{r_0}d\sigma\int_\sigma^{r_0}f_{inn}(\theta(s),t_0)ds+2\int_0^{r_0}f_{in}(\theta(s),t_0)ds\\
=&\int_0^{r_0}(f_{in}(y_0,t_0)-f_{in}(\theta(\sigma),t_0))d\sigma+2\int_0^{r_0}f_{in}(\theta(s),t_0)ds\\
=&r_0f_{in}(y_0,t_0)+\int_0^{r_0}f_{in}(\theta(s),t_0)dt\\
\end{split}
\end{equation}
for $i=1,2,\cdots,n-1$.

Then, by \eqref{eqn-first-1},\eqref{eqn-first-3} and \eqref{eqn-first-4},
\begin{equation}\label{eqn-first-5}
\begin{split}
&\frac{1}{2}E_{f_i^2}(x_0,y_0,t_0)\\
=&\int_0^{r_0}f_{inn}f_i(\theta(s),t_0)dt+\int_0^{r_0} f_{in}^2(\theta(s),t_0)ds\\
=&\int_0^{r_0}f_{inn}(\theta(s),t_0)\int_0^sf_{in}(\theta(\sigma),t_0)d\sigma ds-f_{in}(x_0,t_0)E_{f_i}(x_0,y_0,t_0)+\\
&\int_0^{r_0}f_{in}^2(\theta(s),t_0)ds\\
=&\int_0^{r_0}f_{in}(\theta(\sigma),t_0)d\sigma\int_\sigma^{r_0}f_{inn}(\theta(s),t_0)ds+\int_0^{r_0}f_{in}^2(\theta(s),t_0)ds\\
=&\int_0^{r_0}f_{in}(\theta(s),t_0)(f_{in}(y_0,t_0)-f_{in}(\theta(s),t_0))ds+\int_0^{r_0}f_{in}^2(\theta(s),t_0)ds\\
=&f_{in}(y_0,t_0)\int_0^{r_0}f_{in}(\theta(s),t_0)ds\\
=&-r_0f_{in}(y_0,t_0)^2\leq 0
\end{split}
\end{equation}
for $i=1,2,\cdots,n-1$, and
\begin{equation}\label{eqn-first-6}
\begin{split}
&\frac{1}{2}E_{f_n^2}(x_0,y_0,t_0)\\
=&\int_0^{r_0}f_{nnn}f_n(\theta(s),t_0)ds+\int_0^{r_0} f_{nn}^2(\theta(s),t_0)ds\\
=&\int_0^{r_0}f_{nnn}(\theta(s),t_0)\int_0^s f_{nn}(\theta(\sigma),t_0)d\sigma ds+\int_0^{r_0}f_{nn}^2(\theta(s),t_0)ds\\
=&\int_0^{r_0}f_{nn}(\theta(\sigma),t_0)\int_\sigma^{r_0}f_{nnn}(\theta(s),t_0)dsd\sigma+\int_0^{r_0}f_{nn}^2(\theta(s),t_0)ds\\
=&\int_0^{r_0}f_{nn}(\theta(\sigma),t_0)(f_{nn}(y_0,t_0)-f_{nn}(\theta(\sigma),t_0))d\sigma+\int_0^{r_0}f_{nn}^2(\theta(s),t_0)ds\\
=&f_{nn}(y_0,t_0)E_f(x_0,y_0,t_0)\\
=&\frac{m^2\psi_s\psi}{2}(r_0/2,t_0).
\end{split}
\end{equation}

Moreover, by (1),(4) in Lemma \ref{lm-com} and Theorem \ref{thm-2-oder-1},
\begin{equation}
\begin{split}\label{eqn-second-1}
0\leq&\nabla_{E_i}\nabla_{E_i}Q(x_0,y_0,X_0,t_0)=\frac{\nabla_{E_i}\nabla_{E_i}E_f(x_0,y_0,t_0)}{\psi(r_0/2,t_0)}=\frac{E_{f_{ii}}(x_0,y_0,t_0)}{\psi(r_0/2,t_0)}
\end{split}
\end{equation}
for $i=1,2,\cdots,n-1$.

By (3),(5) in Lemma \ref{lm-com}, \eqref{eqn-first-1}, \eqref{eqn-first-3}, Theorem \ref{thm-1-order-2} and Theorem \ref{thm-2-order-2},
\begin{equation}\label{eqn-second-2}
\begin{split}
0\leq\nabla_{\tilde E_n}\nabla_{\tilde E_n}Q(x_0,y_0,X_0,t_0)=\frac{E_{f_{nn}}(x_0,y_0,t_0)}{\psi(r_0/2,t_0)}-\frac{m\psi_{ss}}{\psi}(r_0/2,t_0)\\
\end{split}
\end{equation}
Furthermore,
\begin{equation}\label{eqn-second-3}
\begin{split}
0\geq\frac{\partial Q}{\partial t}(x_0,y_0,X_0,t_0)=\frac{E_{f_t}(x_0,y_0,t_0)}{\psi(r_0/2,t_0)}-\frac{m\psi_t}{\psi}(r_0/2,t_0)
\end{split}
\end{equation}
Combining \eqref{eqn-second-1},\eqref{eqn-second-2}, and \eqref{eqn-second-3}, and using \eqref{eqn-first-5} and \eqref{eqn-first-6},
\begin{equation}
\begin{split}
0\leq& E_{\Delta f-f_t}(x_0,y_0,t_0)-m(\psi_{ss}-\psi_t)(r_0/2,t_0)\\
=&E_{|\nabla f|^2-q}(x_0,y_0,t_0)-m(\psi_{ss}-\psi_t)(r_0/2,t_0)\\
\leq&\left(m^2\psi_s\psi-m(\psi_{ss}-\psi_t)-\varphi\right)(r_0/2,t_0).\\
\end{split}
\end{equation}
Therefore,
\begin{equation}
m\geq \frac{\psi_{ss}-\psi_t+\sqrt{(\psi_{ss}-\psi_t)^2+4\varphi\psi_{s}\psi}}{2\psi_s\psi}(r_0/2,t_0).
\end{equation}
(2) When $x_0=y_0\in\Omega$, then we know that $X_0$ is the eigenvector of $\nabla^2f(x_0,t_0)$ with minimal eigenvalue. Hence we can choose a orthonormal frame $e_1,e_2,\cdots,e_n=X_0$ such that
\begin{equation}\label{eqn-fij}
f_{ij}(x_0,t_0)=0
\end{equation}
when $i\neq j$ and $m=\frac{2f_{nn}(x_0,t_0)}{\psi_s(0,t_0)}$. It is clear that $$f_{nn}(x_0,t)\leq f_{nn}(x,t_0)$$
 for all $x\in \Omega$ and $t\in [0,T]$. Therefore, the first order necessary condition for minimums implies that
\begin{equation}\label{eqn-fnni}
f_{nni}(x_0)=0\ \mbox{for}\ i=1,2,\cdots,n,
\end{equation}
and
\begin{equation}\label{eqn-fnnt}
\frac{\partial f_{nn}}{\partial t}(x_0,t_0)\leq 0.
\end{equation}
The second order necessary condition for minimums implies that
\begin{equation}\label{eqn-fnnii}
0\leq f_{nnii}(x_0,t_0)\ \mbox{for}\ i=1,2,\cdots,n.
\end{equation}
Moreover, let
$$g(s)=Q(x_0-se_n,x_0+se_n,e_n,t_0).$$
It is clear the $g(0)$ is the minimum of $g$. The Taylor expansion of $g$ is as follows.
\begin{equation}
\begin{split}
g(s)=&\frac{\int_{-s}^{s}f_{nn}(x_0+\sigma e_n,t_0)d\sigma}{\psi(s,t_0)}\\
=&\frac{2f_{nn}(x_0,t_0)+\frac{1}{3}f_{nnnn}(x_0,t_0)s^2+\cdots}{\psi_s(0,t_0)+\frac{1}{6}\psi_{sss}(0,t_0)s^2+\cdots}\\
=&m+\frac{1}{3\psi_s(0,t_0)}(f_{nnnn}(x_0,t_0)-m\psi_{sss}(0,t_0)/2)t^2+\cdots
\end{split}
\end{equation}
So,
\begin{equation}\label{eqn-fnnnn}
0\leq f_{nnnn}(x_0,t_0)-m\psi_{sss}(0,t_0)/2.
\end{equation}
Therefore, by combining \eqref{eqn-fij},\eqref{eqn-fnni},\eqref{eqn-fnnt},\eqref{eqn-fnnii} and \eqref{eqn-fnnnn},
\begin{equation}
\begin{split}
0\leq& (\Delta f-f_t)_{nn}(x_0,t_0)-m\psi_{sss}(0,t_0)/2\\
\leq&  2f_{nn}(x_0)^2-m\psi_{sss}(0,t_0)/2\\
=&m^2\psi_s^2(0,t_0)/2-m\psi_{sss}(0,t_0)/2
\end{split}
\end{equation}
Hence
\begin{equation}
m\geq \frac{\psi_{sss}(0,t_0)}{\psi_{s}^2(0,t_0)}.
\end{equation}

Finally, by noting that
\begin{equation}
\lim_{s\to 0}\frac{\psi_{ss}-\psi_t+\sqrt{(\psi_{ss}-\psi_t)^2+4\varphi\psi_{s}\psi}}{2\psi_s\psi}(s,t)=\frac{\psi_{sss}(0,t)}{\psi_{s}^2(0,t)},
\end{equation}
we get the conclusion.
\end{proof}
As a corollary, we have the following comparison of log-concavity for positive solutions of the heat equation. Similar results presented in different forms can also be found in \cite{AC,Ni}.
\begin{cor}\label{cor-comp-logconcavity}
Let $\Omega$ be a strictly convex domain in $\R^n$ with smooth boundary and  diameter $D$. Let $q(x,t)\in C^\infty(\ol \Omega\times[0,T])$ and $u(x,t)$ be a positive solution of the heat equation
\begin{equation}
\pd{u}{t}-\Delta u+qu=0
\end{equation}
on $\Omega$ with Dirichlet boundary data. Let $f(x,t)=-\log u(x,t)$. Let $\br q(s,t)$ be a smooth even convex function in $s$ on $[-D/2,D/2]\times[0,T]$.
Let $\bar u(s,t)$ be a positive solution of the heat equation
\begin{equation}
\pd{\bar u}{t}-\pd{^2\bar u}{s^2}+\br q\br u=0
\end{equation}
on $[-D/2,D/2]$ with Dirichlet boundary data and $\bar u(s,0)$ strictly log-concave and even in $(-D/2,D/2)$. Let $\bar f(s,t)=-\log\br u(s,t)$. Suppose that
\begin{equation}
E_q(x,y,t)\geq E_{\bar q}(-r(x,y)/2,r(x,y)/2,t)=2\bar q_s(r(x,y)/2,t),
\end{equation}
and
\begin{equation}
E_f(x,y,0)\geq E_{\bar f}(-r(x,y)/2,r(x,y)/2,0)=2\bar f_s(r(x,y)/2,0).
\end{equation}
Then
\begin{equation}
E_f(x,y,t)\geq E_{\bar f}(-r(x,y)/2,r(x,y)/2,t)=2\bar f_s(r(x,y)/2,t)
\end{equation}
for any $x,y\in\Omega$ and $t\in [0,T]$.
\end{cor}
\begin{proof}Because $f(s,t)$ blows up at $s=D/2$, we apply a similar trick in \cite{Ni} by enlarging $D$.

Let $\tilde q$ be a smooth extension of $\bar q$ on $[-D,D]\times [0,T]$ which is still even and convex. Let $\{r_k\}$ be a strictly decreasing sequence of real numbers tending to $1$. Let $u^{(k)}(s,t)$ be the solution of the following boundary value problems:
\begin{equation}
\left\{\begin{array}{l}u^{(k)}_t-u^{(k)}_{ss}+\tilde qu^{(k)}=0\\
u^{(k)}(s,0)=\bar u(s/r_k,0)\\
u^{(k)}(\pm r_k D/2,t)=0.
\end{array}\right.
\end{equation}
Let $f^{(k)}=-\log u^{(k)}$, $\varphi(s,t)=2\bar q_s(s,t)$
and
 $$\psi^{(k)}(s,t)=2f_s^{(k)}(s,t).$$
Note that
\begin{equation}
\psi^{(k)}(s,0)=\frac{2}{r_k}\bar f_s(s/r_k,0)\leq 2\bar f_s(s,0),
\end{equation}
since $\bar f(\cdot,0)$ is even and convex. So,
$$E_f(x,y,0)\geq \psi^{(k)}(r(x,y)/2,0).$$

Then, by Theorem \ref{thm-main} and the equation
$$\psi^{(k)}_{ss}-\psi_t^{(k)}=\psi^{(k)}\psi^{(k)}_s-\varphi$$
on $[0,D/2]$, we know that
$$E_f(x,y,t)\geq \psi^{(k)}(r(x,y)/2,t).$$
By letting $k\to\infty$ in the last inequality, we get the conclusion.
\end{proof}
Applying the last corollary to heat kernel, one can obtain a comparison of log-concavity for heat kernels. The result is also presented in a different form in \cite{Ni}.
\begin{cor}\label{cor-comp-kernel}
Let $\Omega$ be a strictly convex domain in $\R^n$ with smooth boundary and diameter $D$. Let $q(x)\in C^\infty(\ol \Omega)$ and $H(x,z,t)$ be  the heat kernel for
\begin{equation}
\pd{u}{t}-\Delta u+qu=0
\end{equation}
on $\Omega$ with Dirichlet boundary data. Let $\bar q(s)$ be a smooth even convex function  $[-D/2,D/2]$ and
$\ol H(s,\sigma,t)$ be the heat kernel of the heat equation
\begin{equation}
\pd{\bar u}{t}-\pd{^2\bar u}{s^2}+\br q\br u=0
\end{equation}
on $[-D/2,D/2]$ with Dirichlet boundary data. Suppose that
\begin{equation}\label{eqn-comp-heat}
E_q(x,y)\geq E_{\bar q}(-r(x,y)/2,r(x,y)/2)=2\bar q_s(r(x,y)/2).
\end{equation}
Then
\begin{equation}
E_{-\log H(\cdot,z,\cdot)}(x,y,t)\geq E_{-\log\ol H(\cdot,0,\cdot)}(-r(x,y)/2,r(x,y)/2,t)
\end{equation}
for any $x,y,z\in\Omega$ and $t>0$.
\end{cor}
\begin{proof} Let $\{r_k\}$ be a strictly decreasing sequence of real numbers tending to $1$, $q^{(k)}(s)=\frac{1}{r_k^2}\bar q(s/r_k)$ and
\begin{equation}
H^{(k)}(s,t)=\ol H(s/r_k,0,t/r_k^2).
\end{equation}
Then
\begin{equation}
H^{(k)}_{t}-H^{(k)}_{ss}+q^{(k)}H^{(k)}=0.
\end{equation}
Let $\{\delta_k\}$ be a strictly decreasing sequence of real numbers tending to $0$. By the short time asymptotic behavior of heat kernel by Malliavin and Stroock \cite{MS}, for each $\delta_k>0$, there is a $\epsilon_k>0$ small enough, such that
\begin{equation}
\begin{split}
E_{-\log H(\cdot,z,\cdot)}(x,y,\epsilon_k)\geq &E_{-\log \ol H^{(k)}}(-r(x,y)/2,r(x,y)/2,\delta_k)\\
=&2(-\log H^{(k)})_{s}(s,\delta_k)
\end{split}
\end{equation}
for any $x,y\in \Omega$. Moreover, note that
\begin{equation}
E_q(x,y)\geq 2\bar q_s(r(x,y)/2)\geq 2q^{(k)}_s(r(x,y)/2)
\end{equation}
since $q$ is even and convex.

Let $\psi^{(k)}(s,t)=2(-\log H^{(k)})_{s}(s,t+\delta_k)$ and $\varphi^{(k)}(s,t)=2q^{(k)}_s(s).$ Then, by Theorem \ref{thm-main} and that
\begin{equation}
\psi^{(k)}_{ss}-\psi^{(k)}_t=\psi_s^{(k)}\psi^{(k)}-\varphi^{(k)},
\end{equation}
\begin{equation}
\begin{split}
E_{-\log H(\cdot,z,\cdot)}(x,y,t+\epsilon_k)\geq& \psi^{(k)}(r(x,y)/2,t)\\
=&E_{-\log \ol H^{(k)}}(-r(x,y)/2,r(x,y)/2,t+\delta_k).
\end{split}
\end{equation}
Let $k\to \infty$ in the last inequality. We obtain the conclusion.
\end{proof}

By a similar argument as in the proof of Theorem \ref{thm-main}, we have the following elliptic version of Theorem \ref{thm-main}.
\begin{thm}\label{thm-main-elliptic}
Let $\Omega$ be a bounded strictly convex domain in $\R^n$ with smooth boundary and diameter $D$, $q(x)\in C^\infty(\overline \Omega)$ and $u(x)$ be a nonnegative first eigenfunction of the operator $\Delta-q$ on $\Omega$ with Dirichlet boundary data. Let $\psi(s)\in C^\infty([0,D/2])$ with $\psi(0)=0$, $\psi_{ss}(0)=0$ and $\psi_s>0$. Suppose that
\begin{equation}
E_q(x,y)\geq \varphi(r(x,y)/2)
\end{equation}
where $\varphi$ is a nonnegative function on $[0,D/2]$,
Then
\begin{equation}
E_f(x,y)\geq m\psi(r(x,y)/2)
\end{equation}
for any $x,y\in \Omega$, where $f=-\log u$ and
\begin{equation}\label{eqn-m}
m=\min_{s\in [0,D/2]}\frac{\psi_{ss}+\sqrt{\psi_{ss}^2+4\varphi\psi\psi_s}}{2\psi_s\psi}.
\end{equation}
\end{thm}

By similar arguments as in Corollary \ref{cor-comp-logconcavity} or Corollary \ref{cor-comp-kernel}, we can obtain the following log-concavity comparison for first eigenfunctions.
\begin{cor}
Let $\Omega$ be a bounded strictly convex domain in $\R^n$ with smooth boundary  and  diameter $D$. Let $q(x)\in C^\infty(\ol \Omega)$, $\phi$ be a nonnegative first eigenfunction of $\Delta-q$ on $\Omega$ with Dirichlet boundary data and $f=-\log\phi$. Let $\bar q(s)$ be a smooth even convex function on $[-D/2,D/2]$ such that
\begin{equation}
E_{q}(x,y)\geq E_{\bar q}(-r(x,y)/2,r(x,y)/2)=2\bar q_s(r(x,y)/2).
\end{equation}
Let $\bar\phi$ be a nonnegative first eigenfunction of $\frac{d^2}{ds^2}-\bar q(s)$ on $[-D/2,D/2]$ with Dirichlet boundary data and $\bar f=-\log\bar\phi$. Then
\begin{equation}
E_f(x,y)\geq E_{\bar f}(-r(x,y)/2,r(x,y)/2)=2\bar f_s(r(x,y)/2).
\end{equation}
\end{cor}

\end{document}